\documentclass[pdflatex,sn-mathphys-num]{sn-jnl}% Math and Physical Sciences Numbered Reference Style

\usepackage{graphicx}%
\usepackage{multirow}%
\usepackage{amsmath,amssymb,amsfonts}%
\usepackage{amsthm}%
\usepackage{mathrsfs}%
\usepackage[title]{appendix}%
\usepackage{xcolor}%
\usepackage{textcomp}%
\usepackage{manyfoot}%
\usepackage{booktabs}%
\usepackage{algorithm}%
\usepackage{algorithmicx}%
\usepackage{algpseudocode}%
\usepackage{listings}%

\usepackage{mathrsfs} 
\usepackage{enumitem,lineno}
\usepackage{graphicx}

\usepackage[T1]{fontenc}
\usepackage{amsfonts,amsmath,amssymb}
\usepackage{tikz,comment}
\usetikzlibrary{positioning, calc, trees, decorations.markings, patterns, arrows, arrows.meta}
\usepackage{diagmac2}

\theoremstyle{thmstyleone}%
\newtheorem{Theorem}{Theorem}[section]
\newtheorem{Lemma}[Theorem]{Lemma}
\newtheorem{Proposition}[Theorem]{Proposition}
\newtheorem{Corollary}[Theorem]{Corollary}
\newtheorem{Conjecture}[Theorem]{Conjecture}
\newtheorem{Problem}[Theorem]{Problem}

\theoremstyle{thmstyletwo}%
\newtheorem{Example}[Theorem]{Example}
\newtheorem{Remark}[Theorem]{Remark}

\theoremstyle{thmstylethree}%
\newtheorem{Definition}[Theorem]{Definition}

\theoremstyle{thmstyletwo}%
\newtheorem{remark}{Remark}%

\theoremstyle{thmstylethree}%

\begin{document}

\title[On the center of distances of finite ultrametric spaces]{On the center of distances of finite ultrametric spaces}


\author[1,2]{\fnm{Oleksiy} \sur{Dovgoshey}}\email{oleksiy.dovgoshey@gmail.com}

\author*[3]{\fnm{Olga} \sur{Rovenska}}\email{rovenskaya.olga.math@gmail.com}

\affil[1]{\orgdiv{Department of Theory of Function}, \orgname{Institute of Applied Mathematics and Mechanics of NASU},  \city{Sloviansk}, \postcode{84100}, \country{Ukraine}}

\affil[2]{\orgdiv{Department of Mathematics and Statistics}, \orgname{University of Turku}, \city{Turku}, \postcode{20014},  \country{Finland}}

\affil*[3]{\orgdiv{Department of Mathematics and Modeling}, \orgname{Donbas State Engineering Academy},  \city{Kramatorsk}, \postcode{84313},  \country{Ukraine}}


\abstract{The center of distances of a metric space $(X,d)$ is the set
\(
C(X)\) of all $t\in \mathbb R^+$ for which the equation $d(x,p)=t$ has a solution for each $p\in X$. We prove the inequality $|C(X)| \le 1 + \lfloor \log_2 n \rfloor$  for all finite ultrametric spaces $(X,d)$ which have exactly $n$  points. It is also shown that for every integer $n \geq 1$ there exists a finite ultrametric space $(Y,\rho)$ such that $|Y| = n$ and
\(
|C(Y)| = 1 + \log_2 \lfloor n \rfloor.
\)
}




\maketitle

\section{Introduction}

The concept of the center of distances was introduced by Wojciech Bielaś, Szymon Plewik and Marta Walczyńska in paper \cite{BPW2018} as follows.
\begin{Definition}\label{zcgh572}
Let $(X,d)$ be a metric space and let $D(X)$ be the distance set of $(X,d)$,
\[
D(X) := \{ d(x,y) : x,y \in X \}.
\]
The {\it center of distances} of $(X,d)$ is a subset $C(X)$ of the {\it distance set} $D(X)$
 defined as
\begin{equation}
    \label{edft34Gb}
C(X):=\{t \in D(X) : \forall p\in X\ \exists x\in X\ d(p,x)=t\}.
\end{equation}
\end{Definition}
 This concept was used in \cite{BPW2018} for a generalization of John von~Neumann Theorem on permutations of two sequences with the same set of cluster points~\cite{Neumann1935}. The centers of
distances of ultrametric
spaces generated by labeled trees are found in \cite{DR2026MDPI}.
Some interesting properties of the center of distances are proved in \cite{Bartoszewicz2018-1,Bartoszewicz2018,Banakiewicz2023,Kula2024Achievement,NPT2025}, but the authors are not aware of any results describing the properties of these centers in the case of arbitrary ultrametric  spaces.

In the present paper we investigate the center of distances of finite ultrametric spaces. 
The main purpose of the paper is to provide  a solution to the following problem.
\begin{Problem}\label{pr1}
Let $n$ be a positive integer number. Find the smallest integer $M(n)$ satisfying the inequality
\[
|C(X)| \leq M(n)
\]
for each finite ultrametric spaces $(X,d)$ which have at most $n$ points.
\end{Problem}

The paper is organized as follows. The next section contains some definitions and facts from the theory of metric spaces and graph theory. Theorem~\ref{ter1} of Section \ref{sec3} gives a solution of Problem~\ref{pr1} and it is the main result of the paper. In the final Section \ref{sec4} we formulate three conjectures connected with the center of distances of finite ultrametric spaces.

\section{Preliminaries}

Let us start from the basic concept of metric space.

In what follows we will use the symbol $\mathbb R^+$ to denote the set $[0,+\infty)$.

\begin{Definition}\label{deff1}Let $X$ be a non-empty set. A {\it metric} on $X$ is a symmetric function $d\colon X\times X \to \mathbb R^+$ satisfying the equality 
$d(x,x)=0$ for each $x\in X$, the inequality $d(x,y)>0$ for all distinct $x,y\in X$, and  the triangle inequality
 \[d(x,y)\leq d(x,z) + d(z,y)\]
 for all \(x\), \(y \), \(z \in X\).
\end{Definition}

A metric space \((X, d)\) is called an \emph{ultrametric space} if the \emph{strong triangle inequality}
\[
d(x,y)\leq \max \{d(x,z),d(z,y)\}
\]
holds for all \(x\), \(y\), \(z \in X\). In this case, the metric \(d\) is  called an ultrametric on \(X\).

\begin{Definition}\label{eddrr}
Let $(X,d)$ and $(Y,\rho)$ be metric spaces. A  mapping
\(
\Phi \colon X \to Y
\)
is said to be an {\it isometry} if $\Phi$ is bijective and the equality
\[
d(x,y) = \rho(\Phi(x), \Phi(y))
\]
holds for all $x,y \in X$. 
\end{Definition}

Two metric spaces are \emph{isometric} if there exists an
isometry of these spaces.

Let $A$ be a non-empty subset of a metric space $(X,d)$. The quantity
\begin{equation}\label{lokias}
    \operatorname{diam} A := \sup\{ d(x,y) : x,y \in A \}
\end{equation}
is the \emph{diameter} of $A$.  The \emph{open ball} with a \emph{radius} \(r > 0\) and a \emph{center} \(c \in X\) is the set
\begin{equation}
    \label{02}
B_r(c): = \{x \in X \colon d(c, x) < r\}.
\end{equation}

The following proposition claims that every point of an arbitrary open ball  in $(X,d)$ is a center of that ball if $(X,d)$ is ultrametric.

\begin{Proposition}\label{typs}
Let $(X,d)$ be an ultrametric space. 
Then the equality
\begin{equation*}
B_r(c) = B_r(a)
\end{equation*}
holds for each open ball \( B_r(c) \subseteq X \) and every point \( a \in B_r(c) \).
\end{Proposition}

\begin{proof} It follows from  Proposition 18.4 of book \cite{Sch1985}. 
\end{proof}

Let us recall some definitions and facts.

A \textit{graph} is a pair $(V, E)$ consisting of a non-empty set $V$ and a set $E$ whose elements are unordered pairs $\{u, v\}$ of different points $u, v \in V$. For a graph $G = (V, E)$, the sets $V = V(G)$ and $E = E(G)$ are called \textit{the set of vertices} and \textit{the set of edges}, respectively. A graph $G$ is \textit{finite} if $V(G)$ is a finite set. If $\{x, y\} \in E(G)$ holds, then the vertices $x$ and $y$ are called \textit{adjacent} and we say that $x$, $y$ are the {\it ends} of the edge $\{x,y\}$. In what follows, we will always assume that $E(G) \cap V(G) = \emptyset$.

Let $v$ be a vertex of a graph $G$.
The cardinal number of the set
\[
\{u \in V(G) : \{u,v\} \in E(G)\}
\]
is the {\it degree} of $v$. The degree of $v$ will be denoted by $\delta_G(v)$,
\begin{equation}
    \label{zbh0}
\delta_G(v) := \left|\{u \in V(G) : \{u,v\} \in E(G)\}\right|.
\end{equation}

Let $G$ be a graph.
A graph \(G_1\) is a \emph{subgraph} of \(G\) if
\[
V(G_1) \subseteq V(G) \quad \text{and} \quad E(G_1) \subseteq E(G).
\]
In this case, we will write \(G_1 \subseteq G\).

A \emph{path} is a finite graph \(P\) whose vertices can be numbered without repetitions so that
\begin{equation}\label{e3.3-1}
V(P) = \{x_1, \ldots, x_k\} \quad \text{and} \quad E(P) = \{\{x_1, x_2\}, \ldots, \{x_{k-1}, x_k\}\}
\end{equation}
with \(k \geqslant 2\). We will write \(P = (x_1, \ldots, x_k)\) or \(P = P_{x_1, x_k}\) if \(P\) is a path satisfying \eqref{e3.3-1} and {say} that \(P\) is a \emph{path joining \(x_1\) and \(x_k\)}. A graph \(G\) is \emph{connected} if, for every two distinct vertices of \(G\), there is a path \(P \subseteq G\) joining these vertices.

A finite graph $C$ is a \textit{cycle} if there exists an enumeration of its vertices without repetitions such that $V(C) = \{x_1, \ldots, x_n\}$ and
\begin{equation*}
E(C) = \{\{x_1, x_2\}, \ldots, \{x_{n-1}, x_n\}, \{x_n, x_1\}\} \quad \text{with } n \geq 3.
\end{equation*}

\begin{Definition}\label{szaskk}
A connected graph  without cycles is called a \textit{tree}. 
\end{Definition}

In what follows, we only consider the finite trees that have at least two distinct vertices.

\begin{Proposition}
\label{jgds}
Let $G$ be a connected finite graph with $|V(G)|\geq 2$. Then the following conditions are equivalent.
\begin{enumerate}
    \item[(i)] $G$ is a tree.

        \item[(ii)] The equality 
\begin{equation*}
    |V(G)|=1+|E(G)|
\end{equation*}
holds.

            \item[(iii)] Any two distinct vertices of $G$ are joined by a unique path in $G$.
\end{enumerate}
\end{Proposition}
For a proof, see, for example, Theorem 1.5.2 and Corollary 1.5.3 in \cite{Diestel2017}.

\begin{Definition}
    \label{rguij}
Let $T$ be a tree. A vertex $v\in V(T)$  is  called a {\it leaf}  of $T$ if 
\[
\delta_T(v)=1.
\]

If a vertex $v$ is not a leaf of $T$, then we say that $v$ is an {\it internal node} of $T$.
\end{Definition}

A tree $T$ may have a distinguished vertex $r$ called the {\it root}; in this case $T$ is called a {\it rooted tree} and we write $T=T(r)$.

Let $T=T(r)$ be a rooted tree and let $v\in V(T)$. Similarly \cite{Dov2019pNUAA,DP2018pNUAA,DP2019PNUAA,DP2020pNUAA,DPT2015,DPT2017FPTA}, we will denote by $\delta_T^{+}(v)$ the {\it out-degree} of $v$,
\begin{equation}
    \label{zbh1}
\delta_T^{+}(v):=
\begin{cases}
\delta_T(v), & \text{if } v=r,\\
\delta_T(v)-1, & \text{if } v\ne r,
\end{cases}
\end{equation}
where $\delta_T(v)$ is the degree of $v$ defined by \eqref{zbh0} with $G=T$.

\begin{Remark}
    \label{efgyh}
Definition \ref{rguij} and equalities \eqref{zbh0}, \eqref{zbh1} imply that $r$ is a leaf of a rooted tree $T=T(r)$ if and only if $\delta_T^{+}(r)= 1$. Moreover, for the case when $v \in V(T)$ and $v \neq r$, the vertex $v$ is a leaf of $T(r)$ if and only if 
$\delta_T^{+}(v)=0$.
\end{Remark}

\begin{Definition}\label{saqh}
Let $T=T(r)$ be a rooted tree and let $u,v$ be distinct vertices of $T$. 
The vertex $u$ is called the {\it direct successor} of $v$ if
\(
u \in V(T)\setminus\{r\},\) \( \{u,v\}\in E(T),
\)
and
\(
v \in V(P_{u,r}),
\)
where $P_{u,r}$ is the path joining the vertex $u$ with the root $r$ in $T$.
\end{Definition}

\begin{Remark}\label{mnbffss} Let $T=T(r)$ be a rooted tree. Condition $(iii)$ of Proposition \ref{jgds} shows that for each vertex $u\in V(T)\setminus\{r\}$ there is the unique vertex $v\in V(T)$ such that $u$ is a direct successor of $v$.
\end{Remark}

Using the concept of direct successor, we can define the notion of the {\it level} of the vertices of a rooted tree by the following rule.

Let $T =T(r)$ be a rooted tree. Then, by definition, a vertex $u\in V(T)\setminus\{r\}$ has the first level if $u$ is a direct successor of $r$. A vertex $u \in V(T)\setminus\{r\}$ has the second level if $u$ is a direct successor of a vertex of the first level.

In general case, a vertex $u \in V(T)\setminus\{r\}$ has the level $n \geq 2$ if there is a vertex $v \in V(T)$ such that $v$ has the level $n-1$ and $u$ is a direct successor of $v$.

\begin{Definition} \label{outff} A {\it labeled tree} $T=T(l)$ is a tree $T$ together with a labeling $l:V(T)\to \mathbb{R}^{+}$. A {\it labeled rooted tree} $T(r,l)$  is a rooted tree $T(r)$ endowed with given labeling $l$.
\end{Definition}

Thus we assume that the labels on the vertices of trees are some non-negative numbers.

\begin{Definition}\label{odwd}
Let $G$ and $H$ be graphs. A bijection $f: V(G) \to V(H)$ is called an {\it isomorphism}
of $G$ and $H$ if the logical equivalence
\[
(\{u,v\}\in E(G)) \iff (\{f(u),f(v)\}\in E(H))
\]
is valid for all $u,v\in V(G)$. Two graphs are {\it isomorphic} if there
exists an isomorphism of these graphs. 
\end{Definition}

The above definition needs to be modified if the graphs are provided with any additional structure.

\begin{Definition}\label{qwwmgg}
Let $T_1=T_1(r_1,l_1)$ and $T_2=T_2(r_2,l_2)$ be labeled rooted trees. A mapping $f: V(T_1)\to V(T_2)$  is called an 
{\it isomorphism of $T_1(r_1,l_1)$ and $T_2(r_2,l_2)$} if $f$ is an isomorphism of the free (unrooted, without labelings) trees $T_1$ and $T_2$ and, in addition, we have 
\begin{equation*}
    f(r_1)=r_2, \quad l_1(v)=l_2(f(v))
\end{equation*}
for each
$v\in V(T_1)$.
 The labeled rooted trees 
are isomorphic if there is an isomorphism of these trees.
\end{Definition}

The following concept is well-known  (see, for example, \cite[p.~17]{Diestel2017}). 

\begin{Definition}\label{scfr} Let $G$ be a finite graph and let $k \geq 2$ be an integer number. The graph $G$
is called {\it complete $k$-partite} if the vertex set $V(G)$ can be partitioned into $k$ non-empty
disjoint parts, in such a way that no edge has both ends in the same part and
any two vertices of different parts are adjacent.
\end{Definition}

The next definition was introduced in~\cite{PD2014JMS} for characterization of finite ultrametric spaces $(X,d)$ for which the Gomory--Hu inequality~\cite{GH1961S},
\begin{equation*}
    |D(X)|\leq |X|,
\end{equation*}
turns to the equality
\begin{equation*}
    |D(X)|=|X|.
\end{equation*}

\begin{Definition}\label{dvhj} Let $(X,d)$ be a finite metric space with $|X| \geq 2$. A graph  $G_{X}$ is called the {\it diametrical graph} of $(X,d)$ if  
\[
V(G_{X}) = X
\]
 and 
\begin{equation}
    \label{ojhr}
(\{u,v\} \in E(G_{X})) \iff (d(u,v) = \operatorname{diam} X)
\end{equation}
for all $u, v \in X$.
\end{Definition}

The following theorem  directly follows from Theorem 3.2 of \cite{DDP2011pNUAA} and Proposition 3.3 of \cite{BDK2022TAG}.

\begin{Theorem}\label{ref} Let $(X,d)$ be a finite ultrametric space with $|X| \geq 2$. Then  the  diametrical graph $G_{X}$ is complete $k$-partite with $k\geq 2$. Moreover, every part of the diametrical graph $G_{X}$ is an open ball in $(X,d)$ with the radius $r=\operatorname{diam} X$. Conversely, every open ball $B_r(c)\subseteq X$ with $r=\operatorname{diam} X$ and $c\in X$ is a part of $G_{X}$.
\end{Theorem}

For every finite ultrametric space $(X,d)$ with $|X|\geq 2$ we can associate a labeled rooted tree
$T_X = T_X(r_X,l_X)$ such that $r_X:= X$ and $l_X : V(T_X) \to \mathbb{R}^{+}$ is defined by the following rule (see~\cite{PD2014JMS}).

According to Theorem~\ref{ref} the diametrical graph $G_X$ is a complete $k$-partite graph with $k\geq 2$. Let $X_1,\ldots,X_k$ be the parts of $G_X$.
Then, by definition, the nodes of the first level of $T_X$ are the sets
$X_1, \ldots, X_k$ and we define the labels of these nodes as
\begin{equation}\label{gfuika}
l_X(X_i): = \operatorname{diam} X_i,
\end{equation}
$i = 1,\ldots,k$. 
The nodes of the first level with the label $0$ are leaves, and those indicated by strictly
positive labels are internal nodes of the tree $T_X$. If all $X_1, \ldots, X_k$ are leaves,
then the tree $T_X$ is constructed. Otherwise, by repeating the above-described procedure
with the internal nodes $X_j$, $j\in \{1,\ldots,k\}$, instead of $X$, we obtain the nodes of the second level, etc.
Since $X$ is finite, all vertices on some level will be leaves, and the construction of
$T_X$ is completed.

The above-constructed labeled rooted tree $T_X$ is called the \textit{representing tree}
of the ultrametric space $(X,d)$.

\begin{Lemma}\label{swrgh}
Let $(X,d)$ be a finite ultrametric space, let $T_X=T_X(r_X,l_X)$ be the representing tree of $(X,d)$, and let $x_1$ and $x_2$ be two distinct points of $X$.
If $(\{x_1\}, v_1, \ldots, v_n, \{x_2\})$ is the path joining the leaves $\{x_1\}$ and $\{x_2\}$ in $T_X$, then the equality
\begin{equation}\label{sant}
d(x_1,x_2)=\max_{1\le i\le n} l_X(v_i)
\end{equation}
holds.
\end{Lemma}

The proof of Lemma \ref{swrgh} is completely similar to the proof of Lemma~3.2 from~\cite{PD2014JMS}.

\begin{Example}\label{kar92}
Let us consider the four-point ultrametric spaces $(X_4,d_4)$, $(Y_4,\rho_4)$ and labeled rooted trees $T_{X_4}$, $T_{Y_4}$ depicted in Figure~\ref{fig1}. Then $T_{X_4}$, $T_{Y_4}$ are the representing trees of $(X_4,d_4)$, $(Y_4,\rho_4)$ and, moreover, the sets $\{0,3\}$ and $\{0,2,3\}$ are the centers of distances of the spaces 
$(X_4,d_4)$ and $(Y_4, \rho_4)$,
\begin{equation*}
    C(X_4)=\{0,3\}, \quad C(Y_4)=\{0,2,3\}.
\end{equation*}
\end{Example}

\begin{figure}[ht]
\centering
\begin{tikzpicture}[remember picture]

\begin{scope}[xshift=-0.5cm]

\node[draw=none] at (2.7,2.7) {$3$};
\node[draw=none] at (0.5,3.2) {$(X_4,d_4)$};

\node[draw, circle, fill=black, inner sep=1pt] (A) at (0,0) {};
\node[draw, circle, fill=black, inner sep=1pt] (B) at (3.2,-0.3) {};
\node[draw, circle, fill=black, inner sep=1pt] (C) at (-0.1,1.5) {};
\node[draw, circle, fill=black, inner sep=1pt] (D) at (3,2) {};

\draw (A) -- (B) node[midway, below] {$3$};
\draw (A) -- (C) node[midway, left] {$1$};
\draw (A) -- (D) node[midway, right] {$3$};
\draw (C) -- (D) node[midway, above] {$3$};
\draw (B) -- (D) node[midway, right] {$2$};
\draw (B) to[out=45,in=45,looseness=2] (C);

\begin{scope}[xshift=6cm]

\node[draw=none] at (2.7,2.7) {$3$};
\node[draw=none] at (0.5,3.2) {$(Y_4,\rho_4)$};

\node[draw, circle, fill=black, inner sep=1pt] (A) at (0,0) {};
\node[draw, circle, fill=black, inner sep=1pt] (B) at (3.8,0.3) {};
\node[draw, circle, fill=black, inner sep=1pt] (C) at (-0.5,1.6) {};
\node[draw, circle, fill=black, inner sep=1pt] (D) at (3,2) {};

\draw (A) -- (B) node[midway, below] {$3$};
\draw (A) -- (C) node[midway, left] {$2$};
\draw (A) -- (D) node[midway, right] {$3$};
\draw (C) -- (D) node[midway, above] {$3$};
\draw (B) -- (D) node[midway, right] {$2$};
\draw (B) to[out=45,in=45,looseness=1.5] (C);

\end{scope}

\end{scope}

\begin{scope}[xshift=1.4cm,yshift=-2cm,
level distance=1.5cm,
level 1/.style={sibling distance=3cm},
level 2/.style={sibling distance=1.5cm},
every node/.style={circle,draw,minimum size=4mm,inner sep=1pt}
]

\node at (0,0) {3}
    child {node {1}
        child {node {0}}
        child {node {0}}
    }
    child {node {2}
        child {node {0}}
        child {node {0}}
    };

\node[draw=none] at (-1.5,0.4) {$T_{X_4}$};

\begin{scope}[xshift=6cm]

\node at (0,0) {3}
    child {node {2}
        child {node {0}}
        child {node {0}}
    }
    child {node {2}
        child {node {0}}
        child {node {0}}
    };

\node[draw=none] at (-1.5,0.4) {$T_{Y_4}$};

\end{scope}

\end{scope}

\end{tikzpicture}
\vspace{0.7cm}
\caption{The four-point ultrametric spaces $(X_4,d_4)$ and $(Y_4,\rho_4)$ and their representing trees.}
\label{fig1}
\end{figure}
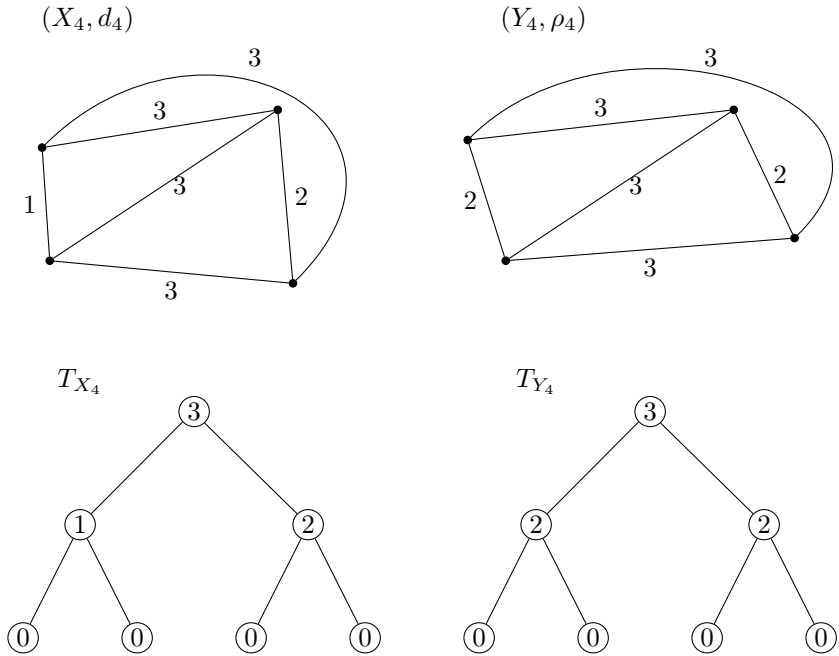

The following theorem gives the necessary and sufficient conditions under which labeled rooted trees are isomorphic to  representing trees of finite ultrametric spaces.

\begin{Theorem}\label{edftt}
Let $T = T(r,l)$ be a finite labeled rooted tree with $|V(T)|\geq 3$. Then the following two conditions are equivalent.
\begin{enumerate}
\item[(i)] For every $v \in V(T)$ we have $\delta_T^{+}(v) \ne 1$, the logical equivalence
\[
(\delta_T^{+}(v)=0) \iff (l(v)=0)
\]
is valid and, in addition, the inequality
\(
l(u) < l(v)
\)
holds whenever $u$ is a direct successor of $v$.
\item[(ii)] There exists a finite  ultrametric space $(X,d)$ with $|X|\geq 2$ such that $T(r,l)$ and the representing tree $T_X=T_X(r_X,l_X)$ are isomorphic as labeled rooted trees.
\end{enumerate}
\end{Theorem}

Theorem \ref{edftt} is a special case of 
Theorem 2.7 of paper \cite{DP2019PNUAA}.

\section{The upper bound of number of points of $C(X)$}\label{sec3}

Let $(X,d)$ be a metric space,  let $p \in X$, and  
let $D_p(X)$ be a subset of 
the distance set $D(X)$
 defined as
\begin{equation}
    \label{fesdi6}
D_p(X) := \{ d(p,x) : x \in X \}. 
\end{equation}

The next proposition is known, see, for example, \cite{NPT2025}.

\begin{Proposition}\label{ugrdd}
Let $(X,d)$ be a metric space and let $C(X)$ be the center of distances of $(X,d)$. Then the equality 
\begin{equation*}
C(X) = \bigcap_{p \in X} D_p(X) 
\end{equation*}
holds.
\end{Proposition}

For finite ultrametric spaces $(X,d)$ we can also describe the sets $D_p(X)$ in terms of representing trees $T_X$.

\begin{Corollary}\label{sefugi}
Let $(X,d)$ be a finite  ultrametric space with $|X|\geq 2$ and let 
\(
T_X=T_X(r_X,l_X)
\)
be the representing tree of \((X,d)\).
Then for every $p\in X$  we have the equality
\begin{equation}
    \label{imgt}
D_p(X)=\{l_X(u) : u \in V(P_{\{p\},r_X})\},
\end{equation}
where \(P_{\{p\},r_X}\) is the path in \(T_X\) joining the leaf \(\{p\}\) with the root \(r_X\).
\end{Corollary}

\begin{proof}
It follows from \eqref{fesdi6} and Lemma \ref{swrgh}.
\end{proof}

The following lemma directly follows from Definition \ref{zcgh572}, Definition \ref{scfr} and Theorem \ref{ref}.

\begin{Lemma}\label{l31}
Let $(X,d)$ be a finite ultrametric space with $|X|\geq 2$. Then the inclusion
\begin{equation*}
\{0,\operatorname{diam} X \} \subseteq C(X)
\end{equation*}
holds.
\end{Lemma}

The next proposition gives an equality expressing the center of distances of  finite ultrametric space $(X,d)$ in terms of centers of distances  of parts of the diametrical graph $G(X)$.

\begin{Proposition}\label{suul}
Let $(X,d)$ be a finite ultrametric space with $|X|\geq 2$ and let the diametrical graph $G(X)$ be complete $k$-partite with the parts $X_1,\ldots,X_k$.
Then the equality
\begin{equation}
    \label{sam1}
C(X)=\{\operatorname{diam} X\}\cup\left(\bigcap_{i=1}^{k} C(X_i)\right)
\end{equation}
holds, where $C(X)$ is the center of distances of the ultrametric space $(X,d)$, and $C(X_1),\dots,C(X_k)$ are the centers of distances of the ultrametric spaces $(X_1,d|_{X_1 \times X_1})$, $\dots$, $(X_k,d|_{X_k \times X_k})$.
\end{Proposition}

\begin{proof}
Equality \eqref{sam1} holds if and only if the inclusions
\begin{equation}
    \label{sam2}
C(X)\supseteq \{\operatorname{diam} X\}\cup\left(\bigcap_{i=1}^{k} C(X_i)\right)
\end{equation}
and
\begin{equation}
    \label{sam3}
C(X)\subseteq \{\operatorname{diam} X\}\cup\left(\bigcap_{i=1}^{k} C(X_i)\right)
\end{equation}
are valid.

Inclusion \eqref{sam2} holds if we have the membership relation
\begin{equation}
    \label{sam4}
t\in C(X)
\end{equation}
whenever
\begin{equation}
    \label{sam5}
t \in \{\operatorname{diam} X\} \cup \left\{ \bigcap_{i=1}^{k} C(X_i) \right\}. 
\end{equation}

Let  $t$ satisfy \eqref{sam5}. If 
\begin{equation*}
    t=\operatorname{diam} X
\end{equation*}
holds, then \eqref{sam4} is valid by Lemma~\ref{l31}. Let us consider the case when 
\begin{equation*}
t \ne \operatorname{diam} X.
\end{equation*}
Then using \eqref{sam5} we obtain the membership relation
\begin{equation}
    \label{sam6}
t \in C(X_i)
\end{equation}
for each  $i\in \{1,\dots,k\}.$
Definition~\ref{zcgh572} and formula \eqref{sam6} imply that the equation
\begin{equation}
    \label{sam7}
d(p,x)=t 
\end{equation}
has a solution $x\in X_i$ for each $p\in X_i$ and every $i\in \{1,\dots,k\}$.

The last statement implies that \eqref{sam7} has a solution $x\in X$ for each $p\in X$, because we have the equality
\[
X=\bigcup_{i=1}^{k} X_i
\]
by Definitions~\ref{scfr}, \ref{dvhj} and Theorem~\ref{ref}. Thus \eqref{sam4} holds by Definition \ref{zcgh572}.

Inclusion \eqref{sam2} follows.

To prove \eqref{sam3} it is enough to
show that \eqref{sam5} holds whenever $t$ satisfies \eqref{sam4}.
Let us consider arbitrary $t$ which satisfies \eqref{sam4}. If $t=\operatorname{diam} X$ holds, then \eqref{sam5} evidently is true. If $t\ne \operatorname{diam} X$ holds, then using \eqref{lokias} and \eqref{sam4} we obtain the strict inequality 
\begin{equation*}
    t<\operatorname{diam} X.
\end{equation*}
The last inequality, Definition~\ref{zcgh572}, equality \eqref{lokias} and formula \eqref{sam4} imply that for each open ball $B_r(p)$ with $r=\operatorname{diam} X$ and $p\in X$ there is a point $x$ such that
\begin{equation*}
    d(p,x)=t \quad \text{and} \quad x\in B_r(p).
\end{equation*}
Since, by Proposition~\ref{typs}, every point of $B_r(p)$ is a center of $B_r(p)$, the number $t$ belongs to the center of distances $C(B_r(p))$ for each open ball $B_r(p)\subseteq X$ whenever $r=\operatorname{diam} X$.
Consequently, we have the membership relation
\begin{equation}
    \label{ss1}
t \in \bigcap_{B\in \mathcal{F}} C(B),
\end{equation}
where $\mathcal{F}$ is a family  of all open balls $B=B_r(p)\subseteq X$ with $p\in X$ and $r=\operatorname{diam} X$.

Now using \eqref{ss1} and Theorem~\ref{ref} we see that \eqref{sam5} holds. 

Inclusion \eqref{sam3} follows.

The proof is completed.
\end{proof}

\begin{Corollary} \label{kvcr}
Let $(X,d)$ be a finite ultrametric space with $|X|\geq 2$, let the diametrical graph $G_X$ be complete $k$-partite with the parts $X_1,\ldots,X_k$
and let $X_{k_0}$ be a part of $G_X$ such that
\[
|X_{k_0}|=\min\limits_{1\leq j\leq k}|X_j|.
\]
Then the inequality
\[
|C(X)| \leq 1+|C(X_{k_0})|
\]
holds, where $C(X)$ is the center of distances of the ultrametric space $(X,d)$ and 
$C(X_{k_0})$ is the center of distances of the ultrametric space 
$(X_{k_0}, d|_{X_{k_0}\times X_{k_0}})$.
\end{Corollary}

\begin{proof}
Equality \eqref{sam1} implies the inclusion
\begin{equation*}
C(X)\subseteq \{\operatorname{diam}X\}\cup C(X_{k_0}),
\end{equation*}
that gives us the inequalities
\[
|C(X)|
\leq \bigl|\{\operatorname{diam}X\}\cup C(X_{k_0})\bigr|
\leq 1+|C(X_{k_0})|,
\]
as required. 
\end{proof}

\begin{Corollary}\label{kgdzw}
Let $(X,d)$ be a finite ultrametric space with $|X|\ge 2$ and let the diametrical graph $G_X$ be a complete bipartite graph with the parts $X_1$ and $X_2$. If $(X_1,d|_{X_1\times X_1})$ and $(X_2,d|_{X_2\times X_2})$ are isometric to an ultrametric space $(Y,\rho)$, then the equality
\begin{equation}
    \label{sxhk}
|C(X)| = 1 + |C(Y)|
\end{equation}
holds, where $C(X)$ and $C(Y)$ are the centers of distances of the spaces $(X,d)$ and $(Y,\rho)$, respectively.
\end{Corollary}

\begin{proof}
Let $(Y,\rho)$ be an ultrametric space such that $(Y,\rho)$, $(X_1,d|_{X_1\times X_1})$ and $(X_2,d|_{X_2\times X_2})$ are isometric. It directly follows from Definitions \ref{eddrr} and \ref{zcgh572} that any two isometric spaces have the same center of distances.
 Hence we have the equality
\[
C(X) = C(Y) \cup \{ \operatorname{diam} X \}
\]
by \eqref{sam1}.
The last equality and the formula 
\begin{equation*}
    \operatorname{diam} X > \operatorname{diam} X_1=\operatorname{diam}X_2= \operatorname{diam} Y
\end{equation*}
give us equality \eqref{sxhk} as required. 
\end{proof}

\begin{Lemma}\label{sabra}
Let $(Y,\rho)$ be a finite ultrametric space with $|Y|\ge 2$ and let $T_Y=T_Y(r_Y,l_Y)$ be the representing tree of $(Y,\rho)$. If $p_1,p_2$ are distinct points of $Y$ and $v_0$ is an internal node of $T_Y$ such that $\{p_1\},\{p_2\}$ are direct successors of $v_0$, then the equality
\begin{equation}
\label{kar1}
D_{p_1}(Y)=D_{p_2}(Y)
\end{equation}
holds, where
\[
D_{p_i}(Y):=\{\rho(p_i,y):y\in Y\} 
\]
for $i=1,2.$
\end{Lemma}

\begin{proof}
Equality \eqref{imgt} implies that \eqref{kar1} holds if and only if we have the equality
\begin{equation}
    \label{kar2}
\{l_Y(u)\colon u \in V(P_{\{p_1\},r_Y})\}=\{l_Y(u)\colon u \in V(P_{\{p_2\},r_Y})\},
\end{equation}
where, for $i=1,2$, $P_{\{p_i\},r_Y}$ is the path 
 joining the leaf $\{p_i\}$ with the root $r_Y$ in $T_Y$. 
 
 The vertices of the path $P_{\{p_1\},r_Y}$ can be numbered without repetitions such that
\[
V(P_{\{p_1\},r_Y})=\{\{p_1\},v_0,v_1,\ldots,v_n,r_Y\}
\]
and
\[
E(P_{\{p_1\},r_Y})=\{\{\{p_1\},v_0\},\{v_0,v_1\},\ldots,\{v_n,r_Y\}\}
\]
(see formula \eqref{e3.3-1}).
Since $\{p_1\}$ and $\{p_2\}$ are direct successors of $v_0$, the nodes $\{p_2\}$ and $v_0$ are adjacent 
\[
\{\{p_2\},v_0\}\in E(T_Y).
\]
Consequently,
\(
(\{p_2\},v_0,v_1,\ldots,v_n,r_Y)
\)
also is  a path in $T_Y$ and this path evidently joins $\{p_2\}$ and $r_Y$ in $T_Y$. Since any two different nodes of an arbitrary tree can be joined by only one path,  the equality
\[
P_{\{v_2\},r_Y}=(\{p_2\},v_0,v_1,\ldots,v_n,r_Y)
\]
holds.
The last equality and the equalities 
\[
l_Y(\{p_1\})=
\operatorname{diam}\{p_1\}=0=\operatorname{diam}\{p_2\}=l_Y(\{p_2\})
\]
(see formula \eqref{gfuika})  imply equality \eqref{kar2}. Equality \eqref{kar1} follows.
\end{proof}

The next proposition shows that every finite ultrametric space $(X,d)$ containing at least two points can be extended while preserving the center of distances $C(X)$.

\begin{Proposition}\label{woorf}
Let $(X,d)$ be a finite ultrametric space with $|X|\ge 2$. Then there exists a finite ultrametric space $(Y,\rho)$ such that
\begin{equation}
    \label{ned}
|Y|=1+|X|
\end{equation}
and  the equality
\begin{equation}\label{gaf1}
C(X)=C(Y)
\end{equation}
holds, where $C(X)$ and $C(Y)$ are the centers of distances of $(X,d)$ and $(Y,\rho)$, respectively.
\end{Proposition}

\begin{proof}
Let $T_X=T_X(r_X,l_X)$ be the representing tree of $(X,d)$. Since $(X,d)$ is finite and $|X|\ge2$ holds, we can find $x_1 \in X$ and an internal node $v_0$ of $T_X$ such that $\{x_1\}$ is a direct successor of $v_0$ (see Remark~\ref{mnbffss}).

Let us define a graph $T^*$ by equalities
\begin{equation}
    \label{avt1}
V(T^*) := V(T_X)\cup\{v^*\},
\end{equation}
\begin{equation}
    \label{avt2}
E(T^*) := E(T_X)\cup\{v^*,v_0\},
\end{equation}
where $v^*$ is an arbitrary point such that $v^*\notin V(T_X)$.

It is clear that $T^*$ is a connected graph satisfying the equalities
\[
|V(T^*)| = 1 + |V(T_X)|,
\]
\[
|E(T^*)| = 1 + |E(T_X)|.
\]
Hence $T^*$ is a tree by Proposition \ref{jgds}.

Now we determine a root $r^*$ in the tree $T^*$ as
\begin{equation}
    \label{a1}
r^* := r_X,
\end{equation}
where $r_X$ is the root of $T_X$, and define a labeling
\(
l^* : V(T^*) \to \mathbb{R}^+
\)
by the formula
\begin{equation}
    \label{a2}
l^*(u) :=
\begin{cases}
l_X(u), & \text{if } u \in V(T_X),\\
0, & \text{if } u = v^*.
\end{cases}
\end{equation}
Equalities \eqref{avt1}, \eqref{avt2} imply the equalities
\begin{equation}
    \label{be1}
\delta_{T^*}(v^*) = 1,\quad
\delta_{T^*}(v_0) = 1 + \delta_{T_X}(v_0) 
\end{equation}
and the equality
\begin{equation}
    \label{be2}
\delta_{T_X}(v) = \delta_{T^*}(v)
\end{equation}
for each $v \in V(T_X)\setminus\{v_0\}$.
Hence, using \eqref{zbh1} we can rewrite \eqref{be1}, \eqref{be2} as
\begin{equation}
    \label{be3}
\delta^{+}_{T^*}(v^*) = 0,
\quad
\delta^{+}_{T^*}(v_0) = 1 + \delta^{+}_{T_X}(v_0), 
\end{equation}
and
\begin{equation}
    \label{be4}
\delta^{+}_{T_X}(v) = \delta^{+}_{T^*}(v)
\end{equation}
for each $v \in V\setminus\{v_0\}$.

Let us prove that condition $(i)$ of Theorem \ref{edftt} is satisfied for $T(r,l)=T^*(r^*,l^*)$.

Since \(T_X(r_X,l_X)\) is the representing tree of \((X,d)\), condition $(i)$ of Theorem~\ref{edftt} and formulas \eqref{a2}, \eqref{be3},  \eqref{be4} show that the relation
\begin{equation}
    \label{be5}
\delta^{+}_{T^*}(v) \neq 1 
\end{equation}
and the logical equivalence
\begin{equation}
    \label{be6}
(\delta^{+}_{T^*}(v)=0)\;\Longleftrightarrow\; (l^{*}(v)=0) 
\end{equation}
are valid.
Thus it suffices to show that the inequality
\begin{equation}\label{eq:star1}
l^{*}(u)<l^{*}(v)
\end{equation}
holds whenever \(u,v\in V(T^{*})\) and \(u\) is a direct successor of \(v\) in the rooted tree \(T^{*}\).

Let \(u\) be the direct successor of \(v\) in \(T^{*}\). Suppose first that \(u,v\in T_X\).
Since \(T_X\) is a subtree of \(T^{*}\) and \[r^{*}=r_X\] holds by \eqref{a1}, Definition~\ref{saqh} implies that \(u\) is the direct successor of \(v\) in \(T_X(r_X)\).
The tree \(T_X=T_X(r_X,l_X)\) is the representing tree of \((X,d)\).
Hence we have the inequality
\[
l_X(u)<l_X(v)
\]
by Theorem~\ref{edftt}. 
This implies \eqref{eq:star1} by \eqref{a2}.

Let us consider the case when the set \(\{u,v\}\) is not a subset of \(V(T_X)\),
\begin{equation}\label{eq:star2}
\{u,v\}\not\subseteq V(T_X).
\end{equation}

Then using \eqref{avt1}, \eqref{avt2} and Definition~\ref{saqh} we obtain the equality
\begin{equation}\label{eq:star3}
\{u,v\}=\{v^{*},v_0\}.
\end{equation}
Since \(v_0\) is a vertex of \(T_X\), formulas \eqref{eq:star2}, \eqref{eq:star3} show that
\begin{equation}\label{eq:star4}
u=v^{*}\quad\text{and}\quad v=v_0.
\end{equation}
Consequently we have the equalities
\begin{equation}\label{eq:star5}
l^{*}(u)=0
\qquad\text{and}\qquad
l^{*}(v)=l_X(v_0)
\end{equation}
by formula \eqref{a2}.
Moreover, since \(\{x_1\}\in T_X\) is a direct successor of \(v_0\) in the representing tree \(T_X\), we have
\begin{equation}\label{eq:star6}
0= l_X(\{x_1\})<l_X(v_0)
\end{equation}
by definition of $T_X$.

Inequality \eqref{eq:star1} follows from \eqref{a2}, \eqref{eq:star5} and \eqref{eq:star6}.

Thus condition $(i)$ of Theorem~\ref{edftt} also is valid for the labeled rooted tree \(T^*(r^*,l^*)\). Consequently, by condition $(ii)$ of this theorem, we can find an ultrametric space \((Y,\rho)\) such that \(T^*(r^*,l^*)\) is isomorphic to the representing tree $T_Y=T_Y(r_Y,l_Y)$ of the space $(Y,\rho)$.

Let us prove equality \eqref{ned} for this $(Y,\rho)$.

It follows directly from the 
 definition of the representing tree $T_X$ that $\{\{x\}: x \in X\}$ is the set of leaves of this tree. 
Similarly the set $\{\{y\}: y \in Y\}$ is the set of leaves of the representing tree $T_Y$ of the space $(Y,\rho)$. 
Moreover, equalities \eqref{avt1}, \eqref{avt2} imply that the set
\[
\{v^*\} \cup \{\{x\}: x \in X\}
\]
  is the set of leaves of the tree $T^{*}$.
It was shown above that $T^{*}(r^{*},l^{*})$ and $T_Y$ are isomorphic. 
Consequently the sets of leaves of the trees $T_Y$ and $T^{*}(r^{*},l^{*})$ have one and the same cardinality.
Now using the equality
\[
\{v^*\} \cap \{\{x\}: x \in X\}=\varnothing,
\]
we obtain
\[
|Y|
=
\left|\{\{y\}: y\in Y\}\right|
=
\left|\{v^*\}\cup \{\{x\}: x\in X\}\right|
=
\left|\{v^*\}\right|
+
\left|\{\{x\}: x\in X\}\right|
=
1+|X|.
\]
Equality \eqref{ned} follows.

Let us prove equality \eqref{gaf1}.

By Proposition \ref{ugrdd} the equalities 
\begin{equation}\label{gaf2}
C(X) = \bigcap_{p \in X} D_p(X) 
\end{equation}
and
\begin{equation}\label{gaf3}
C(Y) = \bigcap_{p \in Y} D_p(Y) 
\end{equation}
hold when
\begin{equation}\label{gaf4}
 D_p(X) =\{d(x,p)\colon x\in X\}
\end{equation}
for $p\in X$, and
\begin{equation}\label{gaf5}
 D_p(Y) =\{\rho(y,p)\colon y\in Y\}
\end{equation}
for $p\in Y$.

Using Corollary \ref{sefugi} we can rewrite equalities \eqref{gaf4} and \eqref{gaf5} as
\begin{equation*}
D_p(X) = \{ l_X(u) : u \in V(P_{\{p\},X}) \} 
\end{equation*}
and
\begin{equation*}
D_p(Y) = \{ l_Y(u) : u \in V(P_{\{p\},Y}) \},
\end{equation*}
where \(P_{\{p\},X}\) is the path in \(T_X\) joining the leaf \(\{p\}\) with the root \(X\) and \(P_{\{p\},Y}\) is the path in \(T_Y\) joining \(\{p\}\) with \(Y\).

Thus \eqref{gaf2}, \eqref{gaf3} imply the equalities
\begin{equation}\label{j1}
C(X)=\bigcap_{x\in X}\{l_X(u):u\in V(P_{\{x\},X})\}
\end{equation}
and
\begin{equation}\label{j2}
C(Y)=\bigcap_{y\in Y}\{l_Y(u):u\in V(P_{\{y\},y})\}.
\end{equation}

The set $\{\{y\}, y\in Y\}$ is the set of leaves of the representing tree $T_Y$ of the space $(Y,\rho)$.

Since the trees $T^*(r^*,l^*)$ and $T_Y(r_Y,l_Y)$ are isomorphic as labeled rooted trees,
equality \eqref{j2} gives us the equality
\begin{equation}\label{j3}
C(Y)=\bigcap_{v\in L(T^*)}\{l^{*}(u):u\in V(P_{u,r^*})\},
\end{equation}
where $L(T^*)$ is the set of all leaves of $T^*$ and
$P_{u,r^*}$ is the path in $T^*$ joining the leaf $v\in L(T^*)$
with the root $r^*$ of $T^*$.
It follows from \eqref{avt1}, \eqref{avt2} that  the equality
\begin{equation}\label{j4}
L(T^*)=\{v^*\}\cup\{\{x\}:x\in X\}
\end{equation}
holds because $\{\{x\}:x\in X\}$ is the set of all leaves of $T_X$.

Now \eqref{j1}, \eqref{j3} and \eqref{j4} imply
\begin{equation}\label{j5}
C(Y)=C(X)\cap\{l^{*}(u):u\in V(P_{v^*,r^*})\}.
\end{equation}

Since $T^*(r^*,l^*)$ and $T_Y(r_Y,l_Y)$ are isomorphic and the leaves $\{x_1\}$, $v^* \in L(T^*)$ are direct successors of $v_0 \in V(T^*)$,
the equality
\begin{equation}
\{\, l^*(u) : u \in V(P_{v^*,r^*}) \,\}
=
\{\, l^*(u) : u \in V(P_{\{x_1\}},r^*) \,\}
\label{eq:star1-1}
\end{equation}
 holds by Lemma \ref{sabra}.
The definition of $T^*(r^*,l^*)$ implies the equality
\begin{equation}
\{\, l^*(u) : u \in V(P_{\{x_1\},r^*}) \,\}
=
\{\, l_X(u) : u \in V(P_{\{x_1\},X}) \,\},
\label{eq:star2-1}
\end{equation}
where $P_{\{x_1\},X}$ is the path in $T_X$ joining the leaf $\{x_1\}$ with the root $X$.
Equalities \eqref{j1}, \eqref{eq:star1-1} and \eqref{eq:star2-1} give us the inclusion
\begin{equation}
\{\, l^*(u) : u \in V(P_{v^*,r^*}) \,\} \supseteq C(X).
\label{eq:inclusion}
\end{equation}
The last inclusion together with equality \eqref{j5} gives us equality \eqref{gaf1} as required.

The proof is completed.
\end{proof}

For arbitrary positive integer number $n$ we denote by $M(n)$  the smallest positive integer number such that
\begin{equation*}
|C(X)| \le M(n)
\end{equation*}
holds for each  ultrametric space $(X,d)$ with $|X|=n$.

\begin{Corollary}\label{jjgfd}
The inequality
\begin{equation}
M(n_1) \le M(n_2)
\label{eq:monotone}
\end{equation}
holds for all positive integers $n_1,n_2$ which satisfy the inequality $n_1 \le n_2$.
\end{Corollary}

\begin{proof}
If the double inequality $2 \le n_1 \le n_2$ holds, then \eqref{eq:monotone} follows from Proposition \ref{woorf}.

To prove \eqref{eq:monotone} in the case when 
\begin{equation*}
    1 \le n_1 \le n_2.
\end{equation*}
It suffices to note that
\begin{equation*}
C(X)=\{0\}
\end{equation*}
holds whenever $(X,d)$ is a one--point ultrametric space and to use Lemma \ref{l31}.
\end{proof}

The following theorem gives us a solution of Problem \ref{pr1} and it is the main result of the paper.

\begin{Theorem}\label{ter1}
Let $n \ge 1$ be an integer number, 
let $\log_2(n)$ be the binary logarithm of $n$, and let $\lfloor \log_2(n) \rfloor$ be the integer part of $\log_2(n)$.
Then the inequality
\begin{equation}\label{kam1}
|C(X)| \leq 1 + \lfloor \log_2 (n) \rfloor
\end{equation}
holds for each ultrametric space $(X,d)$ with $|X|=n$.
Moreover, there exists an ultrametric space $(Y,\rho)$ such that $|Y|=n$
and
\begin{equation}
    \label{kam2}
|C(Y)| = 1 + \lfloor \log_2(n) \rfloor .
\end{equation}
\end{Theorem}

\begin{proof}
Let $\mathbb{N}$ be the set of all positive integers and let
$M:\mathbb{N}\to\mathbb{N}$ be a function defined by the rule: the equality
\begin{equation}
M(n)=m
\label{abrt}
\end{equation}
holds   if and only if we have
\begin{equation}
|C(X)|\le m
\label{qvle}
\end{equation}
 for all ultrametric spaces $(X,d)$ with $|X|=n$, and there exists
an ultrametric space $(Y,\rho)$ such that
\begin{equation}
|Y|=n,
\label{tmka}
\end{equation}
and
\begin{equation}
|C(Y)|=m.
\label{plre}
\end{equation}

We claim that the equalities
\begin{equation}
M(2^l)=l+1,
\label{hzda}
\end{equation}
and
\begin{equation}
    \label{hzda-1}
    M(2^l-1)=l
\end{equation}
are satisfied for every  $l\in \mathbb N$.

We will prove this claim by induction on $l$.

\textbf{Base case.} First we note that the equalities
\[
M(2^0)=1, \qquad M(2^1)=2
\]
follow from Definition \ref{zcgh572}.

\textbf{Induction hypothesis.} Suppose that the equalities
\begin{equation}
M(2^{l_0})=l_0+1,
\label{wkps}
\end{equation}
and 
\begin{equation}
    \label{wkps-1}
    M(2^{l_0}-1)=l_0
\end{equation}
hold for some $l_0\in \mathbb N$.

\textbf{Induction step.} Let us prove the equalities
\begin{equation}
M(2^{l_0+1})=l_0+2
\label{nqte}
\end{equation}
and
\begin{equation}
    \label{4a}
    M(2^{l_0+1}-1)=l_0+1.
\end{equation}

Let $(X,d)$ be an arbitrary ultrametric space such that
\begin{equation}
|X|=2^{l_0+1}
\label{ydmf}
\end{equation}
and 
let $G(X)$ be the diametrical graph of $(X,d)$. By Theorem \ref{ref} the graph
$G(X)$ is complete $k$-partite with $k\ge2$.
Let $X_1,\dots,X_k$ be the parts of $G(X)$ and let $X_{k_0}$ be a part satisfying the equality
\begin{equation}
|X_{k_0}|=\min\limits_{1\leq j \leq k}|X_j|.
\label{vslq}
\end{equation}
Then, by Corollary \ref{kvcr},  we have the inequality
\begin{equation}
    \label{sas12}
|C(X)| \le 1 + |C(X_{k_0})|,
\end{equation}
where \(C(X_{k_0})\) is the center of distances of the ultrametric space
\(
(X_{k_0}, d|_{X_{k_0} \times X_{k_0}}).
\)

Definition \ref{scfr} implies the equality
\[
|X| = \sum_{j=1}^{k} |X_j|.
\]
The last equality, equality \eqref{vslq} and the  inequality \(k \ge 2\) imply that the inequality
\[
2|X_{k_0}| \le |X|
\]
holds.
Consequently we have
\begin{equation}
    \label{Lihgf}
|X_{k_0}| \le 2^{l_0}
\end{equation}
by \eqref{ydmf}.

Write $n_0 := |X_{n_0}|$.
Then, by definition of the function $M:\mathbb{N}\to\mathbb{N}$, we obtain the inequality
\begin{equation}
    \label{okggd}
|C(X_{k_0})| \le M(n_0). 
\end{equation}
It follows from Corollary \ref{jjgfd} that $M:\mathbb{N}\to\mathbb{N}$ is increasing. 
Consequently inequalities \eqref{sas12}, \eqref{Lihgf} and \eqref{okggd} imply
\begin{equation}\label{lqpom}
|C(X)| \le 1 + |C(X_{k_0})| \le 1 + M(n_0) \le 1 + M(2^{l_0}). 
\end{equation}
Since $(X,d)$ is an arbitrary ultrametric space satisfying condition \eqref{ydmf}, the definition of the function $M:\mathbb{N}\to\mathbb{N}$ and \eqref{lqpom} give us 
\[
M(2^{l_0+1}) \le 1 + M(2^{l_0}).
\]
The last inequality and \eqref{wkps} imply the inequality
\begin{equation}
    \label{wessmm}
M(2^{l_0+1}) \le l_0 + 2. 
\end{equation}
Thus, to prove \eqref{nqte} it is sufficient to show that there is an ultrametric space $(Y,\rho)$ such that
\begin{equation}
    \label{b1}
|Y| = 2^{l_0+1} 
\end{equation}
and
\begin{equation}
    \label{b2}
|C(Y)| = l_0 + 2. 
\end{equation}
Let us do it. By Induction hypothesis, we can find an ultrametric space $(Y_1,\rho_1)$ such that
\begin{equation}
    \label{b3}
    |Y_1|=2^{l_0}
\end{equation}
and
\begin{equation}
    \label{b4}
    |C(Y_1)|=l_0+1.
\end{equation}

Let an ultrametric space $(Y_2,\rho_2)$ be isometric to $(Y_1,\rho_1)$, and let $Y_1$ and $Y_2$ be disjoint sets,
\begin{equation*}
Y_1 \cap Y_2 = \varnothing.
\end{equation*}
Since $(Y_1,\rho_1)$ and $(Y_2,\rho_2)$ are isometric, the equality
\begin{equation}\label{eq:diam_equal}
\operatorname{diam} Y_1 = \operatorname{diam} Y_2
\end{equation}
holds.
Let us consider an arbitrary positive number $t^*$ satisfying the condition
\begin{equation*}
t^* > \operatorname{diam} Y_1 = \operatorname{diam} Y_2,
\end{equation*}
and
define a space $(Y,\rho)$ as
\begin{equation}\label{eq:Y_union}
Y := Y_1 \cup Y_2
\end{equation}
and, for all $x,y \in Y$
\begin{equation}\label{eq:rho_def}
\rho(x,y) :=
\begin{cases}
\rho_1(x,y), & x,y \in Y_1,\\
\rho_2(x,y), & x,y \in Y_2,\\
t^*, & \text{otherwise}.
\end{cases}
\end{equation}

It is easy to prove that $\rho$ is an ultrametric on $Y$ and $t^*$ is the diameter of $(Y,\rho)$,
\begin{equation}\label{eq:diamY}
t^*=\operatorname{diam}Y.
\end{equation}

Let $G$ be the complete bi-partite graph with parts $Y_1$ and $Y_2$. 
Equality \eqref{eq:diamY}, formula \eqref{eq:rho_def} and Definition \ref{dvhj} give us the equality
\[
G=G_Y,
\]
where $G_Y$ is the diametrical graph of $(Y,\rho)$.

Since $(Y_1,\rho_1)$ and $(Y_2,\rho_2)$ are disjoint and isometric, equalities \eqref{eq:diam_equal} and \eqref{eq:Y_union} imply 
\begin{equation}\label{eq:cardinality}
|Y| = |Y_1 \cup Y_2| = |Y_1| + |Y_2| = 2|Y_1|=2\cdot 2^{l_0}=2^{l_0+1}.
\end{equation}

Hence we have
\begin{equation}\label{eq:C_relation}
|C(Y)| = 1 + |C(Y_1)|
\end{equation}
by Corollary \ref{kgdzw}. Now equality \eqref{b2}
follows from \eqref{b4} and \eqref{eq:C_relation}. Thus equality \eqref{nqte} holds.

Let us prove equality \eqref{4a}.
Since \( M:\mathbb{N} \to \mathbb{N} \) is an increasing function and $l_0\in \mathbb N$, the double inequality
\begin{equation}
    \label{wcgqq}
2^{l_0} <2^{l_0+1} - 1 \le 2^{l_0+1} 
\end{equation}
implies 
\begin{equation}
    \label{pioio}
M(2^{l_0}) \le M(2^{l_0+1}-1) \le M(2^{l_0+1}).
\end{equation}

Using \eqref{wkps} and \eqref{nqte}, we can rewrite \eqref{pioio} as
\[
l_0 + 1 \le M(2^{l_0+1}-1) \le l_0 + 2.
\]
Thus if \eqref{4a} is false, then the equality
\begin{equation}\label{1z}
    M(2^{l_0+1}-1)=l_0+2
\end{equation}
holds.
Hence, by definition of the function $M : \mathbb{N} \to \mathbb{N}$, there exists an ultrametric space $(Z,\delta)$ such that
\begin{equation}\label{2z}
|Z| = 2^{l_0+1} - 1
\end{equation}
and
\begin{equation}\label{3z}
|C(Z)| = l_0 + 2,
\end{equation}
where $C(Z)$ is the center of distances of the ultrametric space $(Z,\delta)$.

Let $G(Z)$ be the diametrical graph of $(Z,\delta)$.
By Theorem \ref{ref} the graph
$G(Z)$ is complete $k$-partite with $k\ge2$.
Let $Z_1,\dots,Z_k$ be the parts of $G(Z)$ and let $Z_{k_0}$ be a part satisfying the equality
\begin{equation}
|Z_{k_0}|=\min\limits_{1\leq j \leq k}|Z_j|.
\label{vslq-1}
\end{equation}
Then, by Corollary \ref{kvcr},  we have the inequality
\begin{equation}
    \label{sas12-1}
|C(Z)| \le 1 + |C(Z_{k_0})|,
\end{equation}
where \(C(Z_{k_0})\) is the center of distances of the ultrametric space
\(
(Z_{k_0}, d|_{Z_{k_0} \times Z_{k_0}}).
\)

Definition \ref{scfr} implies the equality
\[
|Z|=\sum_{j=1}^{k}|Z_j|.
\]
The last equality, equality \eqref{vslq-1}
and the inequality $k \ge 2$ give us the inequality
\begin{equation*}
2|Z_{k_0}| \le |Z|,
\end{equation*}
that can be written in the form
\begin{equation}\label{eq:x2}
2|Z_{k_0}| \le 2^{l_0+1}-1
\end{equation}
by \eqref{2z}.
Since $|Z_{k_0}|$ is an integer number, inequality \eqref{eq:x2} implies
\begin{equation}\label{eq:x3}
|Z_{k_0}| \le 2^{l_0}-1.
\end{equation}
By definition of the function $M:\mathbb{N}\to\mathbb{N}$ we have the inequality
\begin{equation}\label{eq:x4}
|C(Z_{k_0})| \le M(|Z_{k_0}|).
\end{equation}
Hence the inequality
\begin{equation}\label{eq:x5}
|C(Z_{k_0})| \le M(2^{l_0}-1)
\end{equation}
follows from \eqref{eq:x3} and \eqref{eq:x4}, since $M:\mathbb{N}\to\mathbb{N}$ is increasing. 
Now \eqref{eq:x5} and \eqref{wkps-1} give us the inequality
\begin{equation}\label{eq:x6}
|C(Z_{k_0})| = l_0.
\end{equation}
Consequently,
\begin{equation}\label{eq:x7}
|C(Z)| \le l_0+1
\end{equation}
holds by \eqref{sas12-1} and \eqref{eq:x6}.
The last inequality contradicts equality \eqref{3z}. Equality \eqref{4a} follows.

The Induction step is completed.

Now we are ready to prove inequality \eqref{kam1}.

Let us consider an arbitrary ultrametric space $(X,d)$ with a given $|X| = n$.
If the number $n$ can be written as
\[
n = 2^l,
\]
where $l$ is an integer number, then the equality
\begin{equation}
    \label{5w}
\lfloor \log_2(n) \rfloor = l
\end{equation}
holds and, consequently \eqref{kam1} turns to the inequality
\begin{equation}
    \label{6w}
|C(X)| \leq 1 + l.
\end{equation}
The last inequality follows from \eqref{hzda} and the definition of the function $M : \mathbb{N} \to \mathbb{N}$. Moreover, this definition implies also the existence of an ultrametric space $(Y,\rho)$ that satisfies \eqref{kam2} and the equality $|Y| = n$.

Let us consider the case when the number $n$ cannot be written in the form $2^l$ with integer $l$. In this case we can find $l\in\mathbb{N}$ such that
\begin{equation}\label{eq:x8}
2^{l} < n \le 2^{l+1}-1.
\end{equation}
The last double inequality implies equality \eqref{5w}, that, as was noted above, gives us \eqref{kam1} in form \eqref{6w}.

To complete the proof it suffices to note that equalities \eqref{hzda}, \eqref{hzda-1} imply the equality 
\begin{equation*}
M(n)=l+1
\end{equation*}
for each $n$ satisfying \eqref{eq:x8}, since $M:\mathbb{N}\to\mathbb{N}$ is increasing. Thus, by the definition of $M:\mathbb{N}\to\mathbb{N}$, we can find an ultrametric space $(Y,\rho)$ such that \eqref{kam2} holds whenever $n$ satisfies \eqref{eq:x8}. 

The proof is completed.
\end{proof}

\begin{Remark}\label{yncx}
    In paper \cite{DR2026MDPI}, the authors conjectured that Theorem \ref{ter1} holds for all finite ultrametric spaces (see Conjecture 3 of \cite{DR2026MDPI}). 
\end{Remark}

The next result can be considered as a dual form of Theorem \ref{ter1}.

\begin{Corollary}\label{wsaasr} Let $m \ge 1$ be an integer number and let $(X,d)$ be a finite ultrametric space such that 
\[
|C(X)| = m,
\]
where $C(X)$ is the center of distances of the space $(X,d)$. Then the set $X$ contains at least $2^{m-1}$ points,
\[
|X| \ge 2^{m-1}.
\]
\end{Corollary}

\begin{Example}\label{exw2} Let $n \geq 1$ be an integer number and let $X$ be the set of all finite binary words $(x_1,\ldots,x_n)$, where $x_i \in \{0,1\}$ for each $i \in \{1,\ldots,n\}$. Then the cardinal number of $X$ is $2^n$, 
\[
|X| = 2^n.
\]
Let us define the distance between words $x=(x_1,\ldots,x_n)$ and $y=(y_1,\ldots,y_n)$ as
\[
d(x,y) := 2^{-m}
\]
if  $m$ is the first place at which the words $x$ and $y$ are different, and write
\[
d(x,y) := 0
\]
if $x=y$.
Then $(X,d)$ is an ultrametric space with the center of distances 
\[
C(X)=\{0\}\cup\left\{\tfrac12,\ldots,\tfrac1{2^n}\right\}.
\]
 Thus the equalities
\[
|C(X)| = 1+n = 1+\log_2 |X|=1+\lfloor \log_2|X| \rfloor
\]
hold.
\end{Example}

\begin{Remark}\label{kartf} The above example and results of paper \cite{Staiger2024} show that the centers of distances of finite ultrametric spaces are also connected with topologies on the set of finite words.
\end{Remark}

\section{Concluding Remarks and Expected Results}\label{sec4}

Below we will use the concept of weakly similar ultrametric spaces.

\begin{Definition} Let $(X,d)$ and $(Y,\rho)$ be ultrametric spaces. A mapping $\Phi : X \to Y$ is a {\it weak similarity} of $(X,d)$ and $(Y,\rho)$ if $\Phi$ is bijective and there is a strictly increasing function $f : D(Y) \to D(X)$ such that the equality
\[
d(x,y) = f\bigl(\rho(\Phi(x),\Phi(y))\bigr)
\]
holds for all $x,y \in X$.
\end{Definition}

If $\Phi: X \to Y$is a weak similarity, then we say that $(X, d)$ and $(Y,\rho)$ are {\it weakly similar}.

\begin{remark}
The concept of weak similarity was introduced in \cite{DP2013AMH} for general semimetric spaces. 
\end{remark}

\begin{Conjecture}
Let $n \in \mathbb{N}$ and let $(X,d)$ and $(Y,\rho)$ be ultrametric spaces such that
\begin{equation}\label{eq:4.1}
|X| = |Y| = 2^n
\end{equation}
and
\begin{equation}\label{eq:4.2}
|C(X)| = n + 1.
\end{equation}
Then the equality
\begin{equation*}
|C(Y)| = |C(X)|
\end{equation*}
holds if and only if $(X,d)$ and $(Y,\rho)$ are weakly similar.
\end{Conjecture}

It is easy to see that isometric ultrametric spaces have one and the same center of distances. The following conjecture claims, in particular, that the converse statement is also valid under some restrictions on the spaces  and their centers of distances.

\begin{Conjecture}
Let $n \in \mathbb{N}$ and let $(X,d)$ and $(Y,\rho)$ be ultrametric spaces satisfying equalities \eqref{eq:4.1} and \eqref{eq:4.2}. Then the equality
\begin{equation*}
C(X) = C(Y)
\end{equation*}
holds if and only if $(X,d)$ and $(Y,\rho)$ are isometric.
\end{Conjecture}

The next conjecture was partially motivated by results of Mateusz Kula on the center of distances of Bernstein subsets of $\mathbb{R}$ \cite{Kula2025}.

\begin{Conjecture} Let $A$ be a finite subset of $\mathbb{R}^+$ such that 
\(
0 \in A.
\)
Then there exists a finite ultrametric space $(X,d)$ that satisfies the equalities
\[
D(X)=C(X) = A,
\]
where $D(X)$ and $C(X)$ are the distance set of $(X,d)$ and, respectively, the center of distances of $(X,d)$. 
\end{Conjecture}

\section*{Declarations}

\subsection*{Funding} This study was funded by grant 367319 from the Research Council of Finland.

\subsection*{Conflict of interest/Competing interests} The authors declare no conflicts of interest.

\subsection*{Ethics approval and consent to participate} Not applicable

\subsection*{Consent for publication} Not applicable

\subsection*{Data availability} Not applicable

\subsection*{Materials availability} Not applicable

\subsection*{Code availability} Not applicable

\subsection*{Author contribution} Conceptualization, O.D.; writing -- original draft preparation, O.D.; methodology, O.D.; formal analysis, O.R.; writing -- review and editing, O.R.; investigation, O.D.; resources, O.R. All authors have read and agreed to the published version of the manuscript.

\bigskip

%\bigskip\noindent
%Springer journals and proceedings: \url{https://www.springer.com/gp/editorial-policies}

%\bigskip\noindent
%Nature Portfolio journals: \url{https://www.nature.com/nature-research/editorial-policies}

%\bigskip\noindent
%\textit{Scientific Reports}: \url{https://www.nature.com/srep/journal-policies/editorial-policies}

%\bigskip\noindent
%BMC journals: \url{https://www.biomedcentral.com/getpublished/editorial-policies}
%\end{flushleft}


\bibliography{bib2020.07}
%%\input bib2020.07.bbl

\end{document}